\documentclass[12pt,a4paper,oneside]{amsart}
\usepackage[a4paper]{geometry}
\usepackage{txfonts}
\DeclareMathAlphabet{\mathcal}{OMS}{cmsy}{m}{n} % but do not change mathcal symbols

\DeclareSymbolFont{bbold}{U}{bbold}{m}{n}
\DeclareSymbolFontAlphabet{\mathbbold}{bbold}

\usepackage{amssymb,amscd,enumerate,mathrsfs,perpage,silence,stmaryrd}

\MakePerPage[2]{footnote}

\newcommand{\vertbar}{\>|\>}
\newcommand{\set}[2]{\ensuremath{\{ #1 \vertbar #2 \}}}
\def\liebrack  {\ensuremath{[\,\cdot\, , \cdot\,]}}

\DeclareMathOperator{\B}{B}
\DeclareMathOperator{\C}{C}
\DeclareMathOperator{\Coker}{Coker}
\DeclareMathOperator{\dcobound}{d}
\DeclareMathOperator{\GF}{\mathsf{GF}}
\DeclareMathOperator{\Hom}{Hom}
\DeclareMathOperator{\Homol}{H}
\DeclareMathOperator{\HL}{HL}
\DeclareMathOperator{\im}{Im}
\DeclareMathOperator{\Ker}{Ker}
\DeclareMathOperator{\Sl}{\mathsf{sl}}
\DeclareMathOperator{\sym}{S}
\DeclareMathOperator{\Z}{Z}

\newtheorem{proposition}{Proposition}
\newtheorem*{corollary}{Corollary}
\newtheorem*{theorem}{Theorem}

\numberwithin{equation}{section}
\begin{document}

\title{Commutative Lie algebras and commutative cohomology in characteristic $2$}

\author{Viktor Lopatkin}
\address[Viktor Lopatkin]{
Laboratory of Modern Algebra and Applications, 
St. Petersburg State University and 
St. Petersburg Department of Steklov Mathematical Institute, 
St. Petersburg, Russia}
\email{wickktor@gmail.com}

\author{Pasha Zusmanovich}
\address[Pasha Zusmanovich]{University of Ostrava, Ostrava, Czech Republic}
\email{pasha.zusmanovich@osu.cz}

\date{July 6, 2019 17:46 CEST}

\thanks{
The research of the first author was supported by the grant of the Government of
the Russian Federation for the state support of scientific research carried out
under the supervision of leading scientists, agreement 14.W03.31.0030 dated 
15.02.2018.
}

\begin{abstract}
We discuss a version of the Chevalley--Eilenberg cohomology in characteristic 
$2$, where the alternating cochains are replaced by symmetric ones.
\end{abstract}

\maketitle

\section*{Introduction}

Define a commutative Lie algebra as a commutative algebra satisfying the
Jacobi identity. While in characteristic $\ne 2$ this definition gives rise to
a very special class of locally nilpotent Jordan algebras (studied in the
literature under the names ``mock-Lie'' and ``Jacobi--Jordan'', see
\cite{mock-lie} and references therein), in characteristic $2$ the picture is
entirely different: this class of algebras lies between ordinary Lie algebras 
(where commutativity is replaced by a stronger alternating property) and Leibniz
algebras (where commutativity is dropped altogether), both inclusions are 
strict. The class of commutative Lie algebras admits a good cohomology theory: 
the cohomology is defined via the standard formula for the differential in the 
Chevalley--Eilenberg complex, with the alternating cochains being replaced by symmetric ones.

Why bother with such curiosity? We give four arguments, roughly in increasing
degree of persuasiveness.

\begin{enumerate}[1)]
\item
From the operadic viewpoint, a ``natural'' class of algebras should be defined
by multilinear identities. Moreover, the class of commutative Lie algebras 
appears naturally in certain algebraic topological and categorical contexts.

\item
The underlying complex based on symmetric cochains, unlike the usual one based
on alternating cochains, does not necessary vanish in degrees larger than the
dimension of the algebra. This situation is similar to those occurring in 
cohomology of Lie superalgebras or Leibniz algebras (in any characteristic), 
opens new possibilities, and poses new interesting questions.

\item
Commutative cohomology provides a new invariant of ordinary Lie algebras.

\item
Commutative cohomology of ordinary Lie algebras appears naturally in some
problems related to classification of simple Lie algebras.
\end{enumerate}

The present note is elucidation of points 2--4 (concerning point 1, see an
interesting recent preprint \cite{etingof} for an operadic context, 
\cite{levine} for an algebraic topological context, and \cite{garcia-martinez}
for a categorical context). While elementary in nature, this elucidation 
captures, in our opinion, some important phenomena peculiar to characteristic
$2$ which will be important in the ongoing classification of simple Lie algebras
in that characteristic.

Before we plunge into our considerations, a few remarks are in order.

\WarningsOff
\begin{enumerate}[$\bullet$]
\WarningsOn
\item
Commutative $2$-cocycles of Lie algebras in arbitrary characteristic
do appear naturally in some circumstances and were considered in \cite{alia-d},
\cite{alia-db}, and \cite{comm2}, but, unlike in characteristic $2$, they
seemingly do not lead to any cohomology theory.

\item
For abelian (i.e., with trivial multiplication) Lie algebras, commutative cohomology may be defined in any
characteristic. An instance of such second-degree cohomology appears in
\cite[\S 5]{low} in the context of calculating structure functions on manifolds
of loops with values in compact hermitian symmetric spaces. It seems to be
worthy to study this cohomology and associated structures further. (A 
more-than-decade-ago promise from \cite{low} to develop a 
``symmetric analogue of Spencer cohomology related with symmetric analogue of 
Cartan prolongations and some Jordan algebras'' remained, so far, unfulfilled).

\item
The phenomenon of appearance of not necessary alternating $2$-cocycles in
characteristic $2$ was noted already in \cite[\S 3.4]{jurman}.

\item
Another interesting (and more sophisticated) versions of cohomology theory of
Lie (super){\allowbreak}al\-geb\-ras attempting to fix deficiencies of the
ordinary cohomology in characteristic $2$ were suggested in 
\cite[\S 3]{leites-div-power}. These versions are based on cochain complex 
defined on the divided powers instead of (super)alternating polynomials, with 
various values of the shearing parameters for each (co)homology theory. It seems
to be interesting to combine the constructions of this note and of 
\cite{leites-div-power}.

\item
Everything here can be dualized to get commutative \emph{homology}. This is left
as an exercise to the reader. We are interested primarily in cohomology, due to
its application in structure theory, as explained in \S \ref{sec-motiv} below.

\end{enumerate}

\section{Definitions}

\subsection{Commutative Lie algebras}
Throughout this note, the ground field $K$ is assumed to be of characteristic
$2$, unless stated otherwise. A \emph{commutative Lie algebra} is an algebra $L$
over $K$ with multiplication $\liebrack$ satisfying the commutative identity
$$
[x,y] = [y,x]
$$
and the Jacobi identity
$$
[[x,y],z] + [[z,x],y] + [[y,z],x] = 0
$$
for any $x,y,z \in L$. The usual Lie-algebraic notions of abelian algebra, 
simple algebra, center, ideal, quotient, derivations, deformations, module 
(including the notions of a trivial, adjoint, and dual module), are carried over commutative
Lie algebras without any modification. When considered as an $L$-module, $K$ is
always understood as a trivial module.

\subsection{A note about terminology}
As noted in the introduction, in characteristic different from $2$, commutative
Lie algebras appeared in the literature under different names, see
\cite{mock-lie} and references therein. Neither of these names (``mock-Lie'',
``Jacobi-Jordan'', ``Jordan algebras of nilindex $3$'', etc.) adequately
reflects the characteristic $2$ situation.

Algebras satisfying the anticommutative identity
$$
[x,y] = -[y,x]
$$
and the Jacobi identity, appeared in \cite{levine}, \cite{garcia-martinez}, and
references therein under the name ``quasi-Lie algebras''. Quasi-Lie algebras in
characteristic $2$ are commutative Lie algebras in our terminology, and
ordinary Lie algebras in all other characteristics.

\subsection{Relation to Lie and Leibniz algebras}
As commutative Lie algebras form a subclass of Leibniz algebras, the
relationships between the classes of commutative and ordinary Lie algebras
follow the already established patterns. The Jacobi identity implies that in any
commutative Lie algebra $L$, the squares $[x,x]$, where $x \in L$, linearly span
the central ideal of $L$, denoted by $L^{sq}$ (cf. \cite[\S 1.10]{loday-pira}, 
where in the case of Leibniz algebras this ideal is denoted by $L^{ann}$). More
generally, $L^{sq}$ acts trivially on any $L$-module $M$. The quotient 
$L/L^{sq}$ is a Lie algebra, and one may study commutative Lie algebras by 
considering corresponding extensions of Lie algebras, like it is done, for example, in \cite{dzhu-abd}.

In particular, in any simple commutative Lie algebra $L$ this ideal vanishes,
and hence $L$ is a Lie algebra. Following \cite{dzhu-abd}, one may consider the
next possible minimal situation concerning ideals: a commutative Lie algebra
will be called \emph{almost simple} if each its proper ideal coincides with
$L^{sq}$. For any almost simple commutative Lie algebra $L$, the quotient
$L/L^{sq}$ is simple. Note that, unlike in Leibniz setting, $L^{sq}$ is central,
so, in the case $L$ is not Lie, $L^{sq}$ is necessarily one-dimensional.

\subsection{Commutative cohomology}
Let $L$ be a commutative Lie algebra and $M$ an $L$-module, with a module action
defined by $\bullet$. A \emph{commutative cohomology} of $L$ with coefficients
in $M$, denoted by $\Homol_{comm}^{\bullet}(L,M)$, is defined as cohomology of 
the cochain complex
$$
0 \to \sym^0(L,M) \overset{\dcobound}{\to} \sym^1(L,M) \overset{\dcobound}{\to}
      \sym^2(L,M) \overset{\dcobound}{\to} \dots
$$
where $\sym^n(L,M)$ is the space of $n$-linear symmetric maps 
$f: \underbrace{L \times \dots \times L}_{n \text{ times}} \to M$, i.e. 
$n$-linear maps satisfying
$$
f(x_{\sigma(1)}, \dots, x_{\sigma(n)}) = f(x_1, \dots, x_n)
$$
for any permutation $\sigma \in S_n$. The differential is defined as
\begin{equation*}
\dcobound \varphi(x_1, \dots, x_{n+1}) =
\sum_{1 \le i < j \le n+1}
\varphi([x_i,x_j],x_1,\dots,\widehat{x_i},\dots,\widehat{x_j},\dots,x_{n+1})
+
\sum_{i=1}^{n+1} x_i \bullet \varphi(x_1,\dots,\widehat{x_i},\dots,x_{n+1}) .
\end{equation*}

Note that this is the usual formula for differential in the Chevalley--Eilenberg
complex of a Lie algebra in characteristic $2$ (i.e., all the signs being
dropped). The cocycles and coboundaries in this complex will be customary 
denoted by $\Z^\bullet_{comm}(L,M)$ and $\B^\bullet_{comm}(L,M)$, respectively.

\subsection{``De quadratum nihilo exaequari''\protect\footnote{
From Henri Cartan laudatio on then occasion of receiving Doctor Honoris Causa 
from the Oxford University.
}}
The equality $\dcobound^2 = 0$ may be established by applying verbatim the same
standard arguments used in the case of the usual Chevalley--Eilenberg
cohomology. Namely, for each $x\in L$ let $i(x)$ be endomorphism of the vector
space $\sym^\bullet(L,M) = \bigoplus_{n \ge 0}\sym^n(L,K)$ which maps 
$\sym^n(L,M)$ to $\sym^{n-1}(L,M)$ by the formula
$$
(i(x)f)(x_1, \dots, x_n) = f(x,x_1, \dots, x_n) ,
$$
and let $\theta$ be the natural representation of $L$ in $\sym^n(L,M)$. Then,
for any $x,y \in L$, the usual Cartan formulas hold:
\begin{align}
\theta(x) i(y) + i(y) \theta(x) &= i([x,y])   \notag             \\
i(x) \dcobound + \dcobound i(x) &= \theta(x)  \label{eq-cartan}  \\
\theta(x) \dcobound &= \dcobound \theta(x)    \notag
\end{align}
from what the desired equality $\dcobound^2 = 0$ follows.

Here is a nice heuristic explanation why this works, due to Alexei Lebedev. In
the proof of the equality $\dcobound^2 = 0$ in the Lie-algebraic (i.e.,
alternating) case, we would need an alternating property, and not merely a 
commutativity, of the Lie algebra bracket, only in the case where the formula 
for $\dcobound^2$ would involve expressions of the form $[u,u]$, where $u$ is 
some expression involving $x_1, \dots, x_{n+1}$. Similarly, the alternating, and
not merely symmetric, property of cochains $\varphi$'s would be required only in
the case where $\dcobound^2$ would involve expressions of the form $\varphi(u,u,\dots)$.
Neither of these is the case, and hence the commutativity of the Lie bracket,
and the symmetricity of cochains is enough.

\subsection{No derived functor?}\label{ss-u}
The similarity with the Chevalley--Eilenberg cohomology, however, does have its
limits: it is interesting to see where the standard proof that the
Chevalley--Eilenberg cohomology is the derived functor of the functor of taking
the module invariants $M \mapsto M^L$ (cf., e.g. \cite[\S 7.7]{weibel}), fails
in the case of commutative cohomology.

First we should find a suitable replacement of the universal enveloping algebra
in the commutative case. As $L^{sq}$ acts trivially on any module, the usual
universal enveloping algebra $U(L/L^{sq})$ should serve the purpose: the
categories of representations of $L$ and of $U(L/L^{sq})$ are the same. Define
the chain complex
\begin{equation}\label{eq-res}
                        \dots \overset{\delta}{\to}
U(L/L^{sq}) \otimes \bigvee^3(L) \overset{\delta}{\to}
U(L/L^{sq}) \otimes \bigvee^2(L) \overset{\delta}{\to}
U(L/L^{sq}) \otimes L         \overset{\delta}{\to}
U(L/L^{sq})                   \overset{\varepsilon}{\to} K \to 0
\end{equation}
where $\bigvee^n(L)$ is the $n$-fold symmetric product of $L$, and $\varepsilon$
is the augmentation map with kernel $U^+(L/L^{sq})$. The differential is defined
exactly by the same formula as in the Lie-algebraic (alternating) case:
\begin{multline*}
\delta\Big(u \otimes (x_1 \vee \dots \vee x_n)\Big)
\\ =
\sum_{1 \le i < j \le n}
u \otimes ([x_i,x_j] \vee x_1 \vee \dots \vee \widehat{x_i} \vee \dots \vee
                                              \widehat{x_j} \vee \dots \vee x_n)
+
\sum_{i=1}^n
ux_i \otimes (x_1 \vee \dots \vee \widehat{x_i} \vee \dots \vee x_n) ,
\end{multline*}
where $u \in U(L/L^{sq})$, and $x_1,\dots,x_n \in L$.

By the same arguments as in the Lie-algebraic case -- involving a version of
Cartan formulas (\ref{eq-cartan}) for the complex (\ref{eq-res}) -- we have
$\delta^2 = 0$. However, the complex (\ref{eq-res}) is not exact, so, unlike in
the Lie-algebraic case, it is not a free resolution of the trivial module $K$.
It is not exact already in the case of abelian $L$ (what, in the Lie-algebraic
case, constitute the Koszul complex and essentially serves as an $E_0$ page of 
the spectral sequence abutting to the homology in the general case): for 
example, the chain $1 \otimes (x \vee x)$, for nonzero $x\in L$, belongs to 
$\Ker\delta$, but not to $\im\delta$, since the latter in the second degree lies
in $U^+(L/L^{sq}) \otimes \sym^2(L)$.

Replacing in the complex (\ref{eq-res}) the symmetric product by the 
``alternating'' one, i.e., by the quotient of the tensor algebra $T^\bullet(L)$
by the ideal generated by elements of the form $x \vee x$, $x \in L$, will not work
either: in characteristic $2$, this ``alternating'' product is isomorphic to the
exterior one, $\bigwedge^\bullet(L)$, and for the finite-dimensional $L$, the so 
obtained complex is finite, while the symmetric cohomology apriori may not 
vanish in an arbitrarily large degree (and it does not vanish indeed in all 
examples computed below).

\subsection{Motivation}\label{sec-motiv}
We have encountered commutative cohomology when started a project of description
of simple finite-dimensional Lie algebras having a Cartan subalgebra of toral
rank $1$, of which \cite{grishkov-zus} is the beginning. In the process, one
need to compute various low-degree cohomology of current Lie algebras, i.e. Lie
algebras of the form $L \otimes A$ where $L$ is a Lie algebra and $A$ is a 
commutative associative algebra, for certain particular instances of $L$ and 
$A$. When one tries to extend the known formulas for such cohomology in 
characteristics $\ne 2,3$ from \cite{low} to the case of characteristic $2$, one
naturally encounters low-degree commutative cohomology of $L$. In 
\cite{grishkov-zus}, where we dealt with the case where $L$ is the 
$3$-dimensional simple algebra, commutative cohomology appear in disguise in
Proposition 2.1. The results of this note will be used in subsequent 
classification efforts of simple Lie algebras in characteristic $2$.

\section{Elementary observations}

\subsection{Cohomology of low degree}

The usual interpretations of low-degree cohomology are trivially carried over
from Lie (and Leibniz) algebras to the commutative Lie case:
$\Homol^0_{comm}(L,M) = M^L$, the module of invariants,
$\Homol^1_{comm}(L,K) \simeq (L/[L,L])^*$,
$\Homol^1_{comm}(L,L)$ coincides with outer derivations of $L$,
$\Homol^2_{comm}(L,M)$ describes equivalent classes of abelian extensions
$$0 \to M \to \cdot \to L \to 0 ,$$
$\Homol^2_{comm}(L,L)$ describes infinitesimal deformations of a commutative
Lie algebra $L$, whereas obstructions to prolongability of infinitesimal
deformations to global ones live in $\Homol^3_{comm}(L,L)$.

In particular, the problem of description of almost simple commutative Lie
algebras reduces to determination of $1$-dimensional central extensions
$0 \to Q^{sq} \to Q \to L \to 0$, and hence to computation of
$\Homol^2_{comm}(L,K)$ of all simple Lie algebras $L$.

For any Lie algebra $L$ defined over a field of characteristic $\ne 2$, there is
a useful exact sequence
\begin{equation}\label{eq-seq-no2}
0 \to \Homol^2(L,K) \to \Homol^1(L,L^*) \to \B(L) \to \Homol^3(L,K)
\end{equation}
which goes back to classical works of Koszul and Hochschild--Serre
(see, for example, \cite[\S 1]{comm2} and references therein). Here
$\Homol^\bullet(L,M)$ is the usual Chevalley--Eilenberg cohomology with 
coefficients in an $L$-module $M$, and $\B(L)$ is the space of symmetric invariant bilinear 
forms on $L$, i.e. symmetric bilinear maps $\varphi: L \times L \to K$ such that
\begin{equation}\label{eq-invar}
\varphi([x,y],z) = \varphi([z,x],y)
\end{equation}
for any $x,y,z \in L$.
In characteristic $2$, however, (\ref{eq-seq-no2}) is no longer true, but we
have instead

\begin{proposition}
For any commutative Lie algebra $L$, there is a short exact sequence
$$
0 \to \Homol_{comm}^2(L,K) \to \Homol_{comm}^1(L,L^*) \to \B_{alt}(L) \to
\Homol_{comm}^3(L,K) .
$$
\end{proposition}

Here $\B_{alt}(L)$ denotes the space of all \emph{alternating} bilinear maps
satisfying (\ref{eq-invar}).

\begin{proof}
The proof repeats the standard arguments used in establishing the exact sequence
(\ref{eq-seq-no2}) or its commutative analog in characteristic $\ne 2$
(see, for example, \cite[Proof of Proposition 1.1]{comm2}).
\end{proof}

\subsection{Relation to Chevalley--Eilenberg and Leibniz cohomology}

The natural inclusion of alternating maps to symmetric ones induces, for any
Lie algebra $L$, $L$-module $M$, and $n\in \mathbb N$, a commutative diagram
$$
\begin{CD}
\C^n(L,M)   @>\dcobound>> \C^{n+1}(L,M)      \\
@VVV              @VVV               \\
\sym^n(L,M) @>\dcobound>> \sym^{n+1}(L,M)
\end{CD}
$$
where $\C^n(L,M)$ is the usual space of alternating cochains, and $\dcobound$ is
the usual Chevalley-Eilenberg differential. This, in its turn, induces the map
\begin{equation}\label{eq-inc-lie}
\Homol^n(L,M) \to \Homol^n_{comm}(L,M) .
\end{equation}

Similarly, the natural inclusion of symmetric maps to all multilinear maps
induces, for any commutative Lie algebra $L$ and an $L$-module $M$, a
commutative diagram
$$
\begin{CD}
\sym^n(L,M)           @>\dcobound>> \sym^{n+1}(L,M)           \\
@VVV                        @VVV                    \\
\Hom_K(L^{\otimes n},M) @>\dcobound>> \Hom_K(L^{\otimes n+1},M)
\end{CD}
$$
Here $\dcobound$ in the bottom row denotes the differential in the Leibniz 
complex. This, in its turn, induces the map
\begin{equation}\label{eq-inc-comm}
\Homol^n_{comm}(L,M) \to \HL^n(L,M) ,
\end{equation}
where $\HL^\bullet(L,M)$ denotes the Leibniz cohomology.

Obviously, for $n=0,1$ the maps (\ref{eq-inc-lie}) and (\ref{eq-inc-comm})
are isomorphisms (there is nothing to ``symmetrize'' or ``alternate'' for
cochains in $0$ or $1$ arguments). For any Lie algebra $L$, any $2$-coboundary
with arbitrary coefficients
\begin{equation}\label{eq-2-cobound}
\dcobound \varphi(x,y) =
\varphi([x,y]) + x \bullet \varphi(y) + y \bullet \varphi(x) ,
\end{equation}
and any $3$-coboundary with trivial coefficients
$$
\dcobound \varphi(x,y,z) = 
\varphi([x,y],z) + \varphi([z,x],y) + \varphi([y,z],x)
$$
is alternating, and hence the map (\ref{eq-inc-lie}) is an embedding for $n=2$,
and for $n=3$ and $M=K$. Similarly, for any commutative Lie algebra $L$, the 
Leibniz $2$-coboundary is given by the same formula (\ref{eq-2-cobound}), and 
hence the map (\ref{eq-inc-comm}) is an embedding for $n=2$. In general, 
however, neither of the maps (\ref{eq-inc-lie}) and (\ref{eq-inc-comm}) is an 
embedding or a surjection.

\subsection{Extension of the base field}

The standard arguments based on the universal coefficient theorem, the same as
in the case of ordinary Che\-val\-ley--Ei\-len\-berg cohomology, imply that the
commutative cohomology does not change under field extension: if $L$ is a 
commutative Lie algebra over a field $K$, and $K \subset K^\prime$ is a field extension, then
$$
\Homol^n_{comm}(L \otimes_K K^\prime, M \otimes_K K^\prime) \simeq
\Homol^n_{comm}(L,M) \otimes_K K^\prime .
$$

\section{The cup product}\label{sec-cup}

For a commutative Lie algebra $L$ over a field $K$, define the bilinear map
\[
 \smile : \sym^\bullet(L,K) \times \sym^\bullet(L,K) \to \sym^\bullet(L,K)
\]
by the formula
\begin{equation}\label{cupproduct}
(\varphi \smile \psi)(x_1,\ldots, x_{p+q}) = 
\sum_{IJ} \varphi(x_{i_1}, \ldots, x_{i_p})\cdot \psi(x_{j_1},\ldots,x_{j_q}) ,
\end{equation}
where the sum is taken over all shuffles, i.e. partitions of the sequence 
$\{1,\ldots, p+q\}$ into two disjoint increasing subsequences 
$I = \{i_1,\ldots, i_p\}$ and $J = \{j_1,\ldots, j_q\}$.

It is obvious that the so defined $\smile$ turns $\sym^\bullet(L,K)$ into a 
(graded) associative ring.

\begin{proposition}
The differential $\dcobound$ is a derivation of the ring $\sym^\bullet(L,K)$ 
with respect to the product $\smile$.
\end{proposition}

\begin{proof}
We need to prove that for any $\varphi \in \sym^p(L,K)$ and 
$\psi \in \sym^q(L,K)$, it holds that
\begin{equation}\label{eq-d}
\dcobound (\varphi \smile \psi) = 
\dcobound \varphi \smile \psi + \varphi \smile \dcobound \psi .
\end{equation}

This is verified by direct computation: we have
\begin{eqnarray*}
(\dcobound \varphi \smile \psi)(x_1,\ldots, x_{p+q}) &=& \sum_{IJ} (\dcobound\varphi)(x_{i_1}, \ldots, x_{i_p}) \cdot \psi(x_{j_1},\ldots,x_{j_q}) \\
&=& \sum_{IJ} \sum_{1 \le r < s \le p}
\varphi([x_{i_r},x_{i_s}],x_{i_1},\dots,\widehat{x_{i_r}},\dots,\widehat{x_{i_s}},\dots,x_{i_p}) \cdot \psi(x_{j_1},\ldots,x_{j_q}),
\end{eqnarray*}
\begin{eqnarray*}
(\varphi \smile \dcobound \psi)(x_1,\ldots, x_{p+q}) &=& \sum_{IJ} \varphi(x_{i_1}, \ldots, x_{i_p}) \cdot (\dcobound\psi)(x_{j_1},\ldots,x_{j_q}) \\
&=& \sum_{IJ} \sum_{1\le l < t \le q} \varphi(x_{i_1}, \ldots, x_{i_p}) \cdot \psi([x_{j_l},x_{j_t}],x_{j_1},\dots,\widehat{x_{j_l}},\dots,\widehat{x_{j_t}},\dots,x_{j_q}) ,
\end{eqnarray*}
and
\begin{eqnarray*}
\Big(\dcobound(\varphi \smile \psi)\Big) (x_1,\ldots, x_{p+q}) 
&=& 
\sum_{1 \le \alpha < \beta \le p+q} 
(\varphi \smile \psi) ([x_\alpha,x_\beta], x_1, \ldots, \widehat{x_\alpha}, \ldots, \widehat{x_\beta},\ldots, x_{p+q})
\\
&=& 
\sum_{IJ} \sum_{1 \le r < s \le p}
\varphi([x_{i_r},x_{i_s}],x_{i_1},\dots,\widehat{x_{i_r}},\dots,\widehat{x_{i_s}},\dots,x_{i_p}) \cdot \psi(x_{j_1},\ldots,x_{j_q})
\\
&+& 
\sum_{IJ} \sum_{1\le l < t \le q} \varphi(x_{i_1}, \ldots, x_{i_p}) \cdot \psi([x_{j_l},x_{j_t}],x_{j_1},\dots,\widehat{x_{j_l}},\dots,\widehat{x_{j_t}},\dots,x_{j_q}) ,
\end{eqnarray*}
and the equality (\ref{eq-d}) follows.
\end{proof}

It is obvious that the derivation $\dcobound$ preserves the grading of 
$\sym^\bullet(L,K)$.

As for any ring with a derivation $D$, the kernel $\Ker D$ is a subring, and the
image of $D$ is an ideal in $\Ker D$, we get:

\begin{corollary}
For any commutative Lie algebra $L$:
\begin{enumerate}[\upshape(i)]
\item
The space $\Z^\bullet_{comm}(L,K)$ of commutative cocycles is a subring of the 
ring $\sym^\bullet(L,K)$.
\item
The space $\B^\bullet_{comm}(L,K)$ of commutative coboundaries is an ideal of 
the ring $\Z^\bullet_{comm}(L,K)$.
\item
The commutative cohomology $\Homol^\bullet_{comm}(L,K)$ is a graded associative
ring with respect to the product $\smile$.
\end{enumerate}
\end{corollary}

\section{Examples}

In this section we compute the commutative cohomology in several interesting 
cases.

\subsection{Abelian algebra}

If $L$ is an abelian (commutative) Lie algebra, the differential in the complex
$\sym^\bullet(L,K)$ vanishes, and $\Homol^n_{comm}(L,K) = \sym^n(L,K)$ for any 
$n$.

\subsection{$1$-dimensional algebra}

Obviously, the $1$-dimensional commutative Lie algebra is abelian (and hence is
a Lie algebra). For any module $M$ over the $1$-dimensional algebra $Kx$,
$[x,x] = 0$, we have 
$\sym^n(Kx,M) = \Hom_K\big((Kx)^{\otimes n},M\big) \simeq M$. The differential 
$\dcobound: \sym^n(Kx,M) \to \sym^{n+1}(Kx,M)$ reduces to 
$\dcobound \varphi(x, \dots, x) = n x \bullet \varphi (x,\dots,x)$, and
hence both Lie commutative and Leibniz complexes are reduced to the complex
$$
0 \to M \overset{0}\to M \overset{x}\to M \overset{0}\to M \overset{x}\to \dots
$$
whose cohomology is
$$
\Homol^n_{comm}(Kx,M) = \HL^n(Kx,M) \simeq \begin{cases}
\Ker(x|_M)    & \text{if } n \text{ is even}     \\
\Coker(x|_M)  & \text{if } n \text{ is odd}.
\end{cases}
$$

\subsection{$2$-dimensional algebra}

Let $L$ be the $2$-dimensional nonabelian Lie algebra with the basis $\{a,b\}$,
$[a,b] = a$. Choose a basis in $\sym^n(L,K)$ consisting of the cochains 
$\chi_{pq}$, $p+q = n$, defined by
\[
\chi_{pq}(\underbrace{a,\ldots, a}_r, \underbrace{b,\ldots,b}_s) = 
\begin{cases}
1  & \text{if } p=r \text{ and } q=s \\
0  & \text{otherwise} .
\end{cases}
\]

We have: 
$$
\dcobound \chi_{pq} (\underbrace{a,\ldots, a}_r, \underbrace{b,\ldots, b}_s) =
rs \> 
\chi_{pq}(\underbrace{a,\ldots, a}_{r-1},a, \underbrace{b,\ldots, b}_{s-1}) ,
$$
and hence
\[
\dcobound \chi_{pq} = p(q+1) \> \chi_{p,q+1}.
\]

It follows that $\B_{comm}^n(L,K)$ has a basis consisting of $\chi_{pq}$, where
$p+q = n$, and both $p,q$ are odd; and $\Z_{comm}^n(L,K)$ has a basis consisting
of $\chi_{pq}$, where $p+q = n$, and either $p$ is even, or $q$ is odd.

Therefore, the cocycles $\chi_{pq}$, where $p+q=n$, and $p$ is even, can be 
chosen as basic cocycles whose representatives span 
$\Homol_{comm}^n(L,K)$.

To determine the cup product in terms of this basis, note that by 
(\ref{cupproduct}),
\[
\chi_{pq} \smile \chi_{rs} 
(\underbrace{a,\ldots, a}_{p+r}, \underbrace{b,\ldots, b}_{q+s}) = 
\binom{p+r}{p} \, \binom{q+s}{q} ,
\]
and hence
\[
\chi_{pq} \smile \chi_{rs} = 
\binom{p+r}{p} \, \binom{q+s}{q} \, \chi_{p+r, q+s}.
\]

In particular,
\[
\chi_{p0} \smile \chi_{0s} = \chi_{ps} ,
\]
which shows that the basic cocycles of the form $\chi_{p0}$ and $\chi_{0s}$ 
generate the whole $\Homol_{comm}^\bullet(L,K)$ as a ring.

\subsection{Heisenberg algebra}

The $(2\ell+1)$-dimensional Lie algebra with basis 
$a,b_1,\ldots, b_\ell,c_1,\ldots,c_\ell$, and multiplication
\[
[b_i,a] = [c_i,a] = 0, \quad [b_i,c_j]  = \begin{cases}
a & \text{ if } i=j \\
0 & \text{ if } i\ne j ,
\end{cases}
\]
is called the \emph{Heisenberg algebra}, and is denoted by $\mathcal H_\ell$.

To compute commutative cohomology of $\mathcal H_\ell$ with coefficients in the
trivial module, we will use algebraic discrete Morse theory, briefly recalled in
Appendix (which should be consulted for all undefined notions and notation in this section). A very similar in spirit computation of the usual 
Che\-val\-ley--Ei\-len\-berg \emph{homology} of the Heisenberg algebra in 
characteristic $2$, was performed earlier in \cite{S2}.

Any cochain $\varphi \in \sym^n(\mathcal H_\ell, K)$ is determined uniquely by 
its values on the basic elements:
\begin{equation}\label{eq-val}
\varphi(\underbrace{a,\ldots,a}_\alpha, \underbrace{b_1,\ldots, b_1}_{\beta_1}, \ldots, \underbrace{b_\ell, \ldots,b_\ell}_{\beta_\ell}, \underbrace{c_1,\ldots, c_1}_{\gamma_1}, \ldots, \underbrace{c_\ell, \ldots,c_\ell}_{\gamma_\ell}),
\end{equation}
where 
\begin{equation}\label{eq-part}
\alpha + \beta_1 +\cdots + \beta_\ell + \gamma_1 + \cdots +\gamma_\ell = n .
\end{equation}

Assuming $\beta = (\beta_1, \dots, \beta_\ell)$ and 
$\gamma = (\gamma_1, \dots, \gamma_\ell)$, the following shorthand notation will
be used: $\varphi(\alpha;\beta;\gamma)$ will denote the corresponding value 
(\ref{eq-val}), and $\alpha + \beta + \gamma$ will denote the left-hand side of (\ref{eq-part}). At the same time, $\beta \pm \beta^\prime$ denotes the
vector obtained by the usual coordinate-wise addition or subtraction of vectors in 
$\mathbb Z_{\ge 0}^\ell$, i.e. $(\beta_1 \pm \beta_1^\prime, \dots, \beta_\ell \pm \beta_\ell^\prime)$, 
similarly for $\gamma$'s. The vector of length $\ell$ having $1$ at the $i$th 
place, and $0$ at all other places, will be denoted by $\mathbbold 1_i$. 
Further, define
\begin{alignat*}{2}
& I_0(\beta) = 
\set{\,i \in \{1,\ldots, \ell\}}{\beta_i \text{ is even}\,&}\phantom{.}
\\
& I_1(\beta) = 
\set{\,i \in \{1,\ldots, \ell\}}{\beta_i \text{ is odd}\,&} .
\end{alignat*}

For any triple $(\alpha; \beta; \gamma)$ such that 
$\alpha + \beta + \gamma = n+1$, we have:
\begin{equation}\label{cobofh}
\dcobound \varphi(\alpha;\beta;\gamma) = 
\sum_{\substack{1 \le i \le \ell \\ \beta_i >0, \gamma_i > 0}}
\beta_i\gamma_i \> 
\varphi(\alpha+1; \beta - \mathbbold 1_i; \gamma - \mathbbold 1_i) .
\end{equation}

Now choose a basis $X_n$ in $\sym^n(\mathcal H_\ell, K)$ consisting of the 
cochains $\chi_{(\alpha; \beta; \gamma)}$, $\alpha + \beta + \gamma = n$, 
defined by
$$
\chi_{(\alpha; \beta; \gamma)}(\tilde\alpha; \tilde\beta; \tilde\gamma) = 
\begin{cases}
1 & \text{ if } 
    \tilde\alpha = \alpha, \tilde\beta = \beta, \tilde\gamma = \gamma \\
0 & \text{ otherwise}.
\end{cases}
$$

The formula (\ref{cobofh}) implies then $\dcobound \chi_{(0;\beta;\gamma)} = 0$,
and 
\[
\dcobound 
\chi_{(\alpha;\beta;\gamma)}
(\alpha-1;\beta + \mathbbold 1_i;\gamma + \mathbbold 1_i) = 
(\beta_i + 1)(\gamma_i+1)
\]
for any $\alpha >0$ and $1 \le i \le \ell$. This, in its turn, implies
\begin{equation}\label{dofHeis}
\dcobound \chi_{(\alpha;\beta;\gamma)} = \begin{cases}
\sum_{i=1}^\ell (\beta_i +1)(\gamma_i +1) \> 
\chi_{(\alpha-1; \beta + \mathbbold 1_i; \gamma + \mathbbold 1_{i})} 
  & \text{ if } \alpha > 0 \\
0 & \text{ if } \alpha = 0 .
\end{cases}
\end{equation}

Now we are in the position to apply algebraic discrete Morse theory to the 
cochain complex \linebreak 
$\big(\sym^\bullet(\mathcal H_\ell,K), \dcobound\big)$. In the graph $\Gamma\big(\sym^\bullet(\mathcal H_\ell,K)\big)$ constructed from this 
complex with the chosen basis $\bigcup_{n\ge 0} X_n$, define the set $M$ 
consisting of all edges of the form
\[
\chi_{(\alpha;\beta;\gamma)} \to
\chi_{(\alpha-1; \beta + \mathbbold 1_k; \gamma + \mathbbold1_k)} , 
\]
where $k = \max \bigl( I_0(\beta)\cap I_0(\gamma) \bigr)$ (so both 
$\beta_k$, $\gamma_k$ are even), and 
$$
k = \max \bigl(I_0(\beta)\cap I_0(\gamma) \bigr) > 
\max \bigl( I_1(\beta) \cap I_1(\gamma) \bigr) .
$$

The set $M$ can be depicted as horizontal arrows in the following graph (where 
it is assumed that $i,j \in I_1(\beta)\cap I_1(\gamma)$):
\[
\begin{CD}
\vdots @. \vdots 
\\
@AAA      @AAA   
\\
\chi_{(\alpha-1; \beta + \mathbbold 1_i; \gamma + \mathbbold 1_i)} 
@>>>
\chi_{(\alpha-2; \beta + \mathbbold 1_k + \mathbbold 1_i; \gamma + \mathbbold 1_k + \mathbbold 1_i)} 
\\
@AAA      @AAA    
\\
\chi_{(\alpha;\beta;\gamma)} 
@>>>
\chi_{(\alpha-1; \beta + \mathbbold 1_k; \gamma + \mathbbold 1_k)} 
\\
@AAA      @AAA
\\
\chi_{(\alpha+1; \beta - \mathbbold 1_j; \gamma - \mathbbold 1_j)} 
@>>>
\chi_{(\alpha; \beta + \mathbbold 1_k - \mathbbold 1_j; \gamma + \mathbbold 1_k - \mathbbold 1_j)} 
\\
@AAA      @AAA
\\
\vdots @. \vdots
\end{CD}
\]

It is clear that after flipping all the horizontal arrows, the new graph 
$\Gamma^M\big(\sym^\bullet(\mathcal H_\ell,K)\big)$ does not contain directed 
cycles. Also, no vertex is incident to more than one edge in $M$. Therefore, $M$ is an 
acyclic matching.

The set of vertices in $V = \bigcup_{n\ge 0} X_n$ which do not serve as a tail 
for any arrow in $M$, is equal to
\begin{equation}\label{eq-s0}
\{\> \chi_{(0;\beta;\gamma)} \>\} \cup 
\set{\>\chi_{(\alpha;\beta;\gamma)}}
{\alpha > 0, \> \max \bigl( I_0(\beta) \cap I_0(\gamma) \bigr) < \max \bigl( I_1(\beta) \cap I_1(\gamma) \bigr) \>} ,
\end{equation}
while the set of vertices which do not serve as a head for any arrow in $M$, is
equal to
\begin{equation}\label{eq-s1}
\set{\> \chi_{(\alpha;\beta;\gamma)}}
{\max \bigl( I_0(\beta) \cap I_0(\gamma) \bigr) > \max \bigl( I_1(\beta) \cap I_1(\gamma) \bigr) \>} .
\end{equation}

Thus $\bigcup_{n\ge 0} X_n^M$, being the intersection of the sets (\ref{eq-s0})
and (\ref{eq-s1}), is equal to the set $\mathscr C_0 \cup \mathscr C_1$, where
\begin{equation*}
\mathscr C_0 = \set{\> \chi_{(0;\beta;\gamma)}}
{\max \bigl( I_0(\beta) \cap I_0(\gamma) \bigr) > 
\max \bigl(I_1(\beta) \cap I_1(\gamma) \bigr) \>} ,
\end{equation*}
and
\begin{equation*}
\mathscr C_1 = \set{\> \chi_{(\alpha;\beta;\gamma)}}
{I_0(\beta) \cap I_0(\gamma) = I_1(\beta) \cap I_1(\gamma) = \varnothing \>} .
\end{equation*}

By (\ref{dofHeis}), all cochains from both $\mathscr C_0$ and $\mathscr C_1$ are
cocycles, and then by Theorem from Appendix A, $\mathscr C_0 \cup \mathscr C_1$
forms a basis of the cohomology $\Homol_{comm}^\bullet (\mathcal H_\ell, K)$. (To be more precise, a basis of 
the $n$th degree cohomology $\Homol_{comm}^n (\mathcal H_\ell, K)$ is formed by
cocycles from $\mathscr C_0$ with $\beta + \gamma = n$, and by cocycles from
$\mathscr C_1$ with $\alpha + \beta + \gamma = n$).

Let us look now at the ring structure of 
$\Homol^\bullet_{comm}(\mathcal H_\ell, K)$. For any two triples 
$(\alpha; \beta; \gamma)$ and $(\alpha^\prime; \beta^\prime; \gamma^\prime)$ we
have:
\begin{equation*}
\chi_{(\alpha; \beta; \gamma)} \smile 
\chi_{(\alpha^\prime; \beta^\prime; \gamma^\prime)} = 
\binom{\alpha + \alpha^\prime}{\alpha} \,
\binom{\beta + \beta^\prime}{\beta}    \,
\binom{\gamma + \gamma^\prime}{\gamma} \>
\chi_{(\alpha + \alpha^\prime; \beta + \beta^\prime; \gamma + \gamma^\prime)} ,
\end{equation*}
where $\binom{\beta^\prime}{\beta}$ is a shorthand for the product 
$\binom{\beta_1^\prime}{\beta_1} \cdots \binom{\beta_\ell^\prime}{\beta_\ell}$,
similarly for $\gamma$'s. From this formula it is clear that 
$\mathscr C_0 \smile \mathscr C_0 \subseteq \mathscr C_0$,
$\mathscr C_0 \smile \mathscr C_1 \subseteq \mathscr C_1$, and
$\mathscr C_1 \smile \mathscr C_1 \subseteq \mathscr C_1$, and therefore, as a 
ring, $\Homol^\bullet_{comm}(\mathcal H_\ell, K)$ is decomposed into
the semidirect sum of two subrings:
$$
\Homol^\bullet_{comm}(\mathcal H_\ell, K) \simeq 
K\mathscr C_0 \inplus K\mathscr C_1 ,
$$
where $K\mathscr C_0$ acts on $K\mathscr C_1$.

\subsection{Zassenhaus algebras}\label{ss-zass}

The algebra $W_1(n)$ is defined as an algebra of special derivations
$\mathcal O_1(n) \partial$ of the divided powers algebra $\mathcal O_1(n)$ (see,
e.g., \cite{dzhu-abd}, \cite{jurman}, or \cite{grishkov-zus} for details). It 
has the basis $\set{e_i = x^{(i+1)} \partial}{-1 \le i \le 2^n-2}$ with 
multiplication
$$
[e_i,e_j] = \begin{cases}
\binom{i+j+2}{i+1} \> e_{i+j} &\text{if } -1 \le i+j \le 2^n-2 \\
0                             &\text{otherwise}.
\end{cases}
$$

In characteristic $2$, unlike in bigger characteristics, the algebra $W_1(n)$ is
not simple, but its commutant $W_1^\prime(n)$ of dimension $2^n-1$, linearly
spanned by elements $\set{e_i}{-1 \le i \le 2^n-3}$, is. The algebras
$W_1^\prime(n)$ are referred as \emph{Zassenhaus algebras}. The basic elements
provide the standard grading
$$
W_1^\prime(n) = \bigoplus_{i=-1}^{2^n-3} K e_i .
$$

In the first nontrivial case $n=2$, the algebra $W_1^\prime(2)$ is
$3$-dimensional, with multiplication table
\begin{equation*}
[e_{-1},e_0] = e_{-1}, \quad [e_1,e_0] = e_1, \quad [e_{-1},e_1] = e_0 ,
\end{equation*}
and is an analog of $\Sl(2)$ in big characteristics.

Another realization of the algebra $W_1(n)$ is defined over the field $\GF(2^n)$
as the algebra with the basis $\set{f_\alpha}{\alpha \in \GF(2^n)}$ and 
multiplication
$$
[f_\alpha, f_\beta] = (\alpha + \beta) f_{\alpha + \beta}
$$
for $\alpha, \beta \in \GF(2^n)$. Again, in characteristic $2$ this algebra is 
not simple, but its commutant $\set{f_\alpha}{\alpha \in \GF(2^n)^*}$, 
isomorphic to $W_1^\prime(n)$, is.

For any $k$ elements $\alpha_1, \dots, \alpha_k \in \GF(2^n)^*$ such that the 
sum of any number of these elements is nonzero, the $2^k - 1$ elements 
$f_{\alpha_{i_1} + \dots + \alpha_{i_\ell}}$, 
$1 \le i_1 \le \dots \le i_\ell \le k$, span a subalgebra
$L(\alpha_1, \dots, \alpha_k)$ of $W_1^\prime(n)$ isomorphic to $W_1^\prime(k)$.

\begin{proposition}
$\Homol^2_{comm}(W_1^\prime(n),K)$ has dimension $n$. The basic cocycles
can be chosen as
\begin{equation}\label{eq-comm2}
e_i \vee e_j \mapsto \begin{cases}
1 &\text{if } i=j=2^k - 2, \text{or } \{i,j\} = \{-1,2^{k+1}-3\} \\
0 &\text{otherwise}.
\end{cases}
\end{equation}
for $k = 0,\dots,{n-1}$.
\end{proposition}

\begin{proof}
It is straightforward to check that the maps (\ref{eq-comm2}) are indeed
commutative $2$-cocycles (that boils down to the fact that if $i,j \ge 0$
and $i+j = 2^k-2$, then $[e_i,e_j] = \binom{2^k}{i+1} \> e_{2^k-2} = 0$). Since 
these cocycles are non-alternating, and $2$-coboundaries are alternating, their
cohomological independence is equivalent to the linear independence, and the 
latter follows from the fact that all they have different weights with respect 
to the standard grading of $W_1^\prime(n)$. Thus 
$\dim \Homol^2_{comm}(W_1^\prime(n),K) \ge n$. To prove that we have here an equality, we will switch to the basis $\{f_\alpha\}$.

We shall prove that the basic cocycles in $\Homol_{comm}^2(W_1^\prime(n),K)$
can be chosen as
\begin{equation}\label{eq-falpha}
f_\alpha \vee f_\beta \mapsto \begin{cases}
\lambda_\alpha &\text{if } \alpha = \beta \\
0              &\text{if } \alpha \ne \beta
\end{cases}
\end{equation}
where $\alpha, \beta \in \GF(2^n)^*$, $\lambda_\alpha \in K$, subject to linear
relations
\begin{equation}\label{eq-rel}
\alpha \lambda_\alpha + \beta \lambda_\beta + (\alpha + \beta) \lambda_{\alpha + \beta} = 0
\end{equation}
for any $\alpha,\beta \in \GF(2^n)^*$, $\alpha \ne \beta$.

We proceed similarly to \cite{alia-db} where, in order to prove the vanishing of
commutative $2$-cocycles on simple classical Lie algebras in characteristic 
$>2$, first the rank $2$ case is established, and the general case follows 
easily.

So, first consider the cases $n=2$ and $n=3$. In that cases the statement 
follows from direct computations, similar to those performed in 
\cite[Theorem 6.5]{alia-d} and \cite{alia-db}.
These computations can be also performed on computer, using a simple GAP
program for computations of the space of commutative $2$-cocycles on a given
Lie algebra (see \cite[footnote at \S 3]{comm2}).

In the general case $n \ge 3$, take arbitrary
$\alpha, \beta, \gamma \in \GF(2^n)^*$, $\alpha + \beta + \gamma \ne 0$,
and restrict an arbitrary cocycle $\varphi \in \Z^2_{comm}(W_1^\prime(n),K)$ to
the $7$-dimensional subalgebra $L(\alpha,\beta,\gamma)$ linearly spanned by
$f_\alpha,\> f_\beta,\> f_\gamma, \> f_{\alpha + \beta}, \>
f_{\alpha + \gamma}, \> f_{\beta + \gamma}, \> f_{\alpha + \beta + \gamma}$.
Obviously, this restriction is a commutative $2$-cocycle on \linebreak
$L(\alpha,\beta,\gamma) \simeq W_1^\prime(3)$, and by the just established case
$n=3$, we have that first,
$$
\varphi(f_\alpha, f_\beta) =
(\alpha + \beta) \omega_{\alpha,\beta,\gamma} (f_{\alpha + \beta})
$$
for some linear map 
$\omega_{\alpha,\beta,\gamma}: L(\alpha,\beta,\gamma) \to K$, and second, that 
the relation (\ref{eq-rel}) holds for 
$\lambda_\alpha = \varphi(f_\alpha, f_\alpha)$.
Embedding the pair $\alpha, \beta$ into another triple
$\alpha, \beta, \gamma^\prime$, we see that $\omega_{\alpha,\beta,\gamma}$
does not depend on $\gamma$. In the same vein, it does not depend neither on
$\alpha$, nor on $\beta$, so
$\varphi(f_\alpha,f_\beta) = \dcobound \omega ([f_\alpha, f_\beta])$ for any
$\alpha, \beta \in \GF(2^n)^*$, $\alpha \ne \beta$, and some linear map
$\omega: W_1^\prime(n) \to K$. Consequently, $\varphi$ can be represented as the
sum of $\dcobound \omega$ and a map of the form (\ref{eq-falpha}). The latter
maps are obviously commutative $2$-cocycles, and we are done.

It remains to determine the dimension of $\Homol^2_{comm}(W_1^\prime(n),K)$.
The relation (\ref{eq-rel}) can be expanded as
$$
\lambda_{\alpha_1 + \dots + \alpha_k} =
\frac{\alpha_1}{\alpha_1 + \dots + \alpha_k} \lambda_{\alpha_1} + \dots +
\frac{\alpha_k}{\alpha_1 + \dots + \alpha_k} \lambda_{\alpha_k}
$$
for any $\alpha_1, \dots, \alpha_k \in \GF(2^n)^*$, 
$\alpha_1 + \dots + \alpha_k \ne 0$, what means that
$\dim \Homol^2_{comm}(W_1^\prime(n),K)$ is equal to the number of the generators
of the additive group of $\GF(2^n)$. The latter number is equal to dimension of
$\GF(2^n)$ as a vector space over $\GF(2)$, and hence is equal to $n$.
\end{proof}

Note that since $\dim \B^2(W_1^\prime(n),K) = \dim W_1^\prime(n) = 2^n-1$, we
have $\dim \Z^2_{comm}(W_1^\prime(n),K) = 2^n + n - 1$.

Note also, that the calculations above imply that every alternating $2$-cocycle
on $W_1^\prime(n)$ is a coboundary, and hence $\Homol^2(W_1^\prime(n),K) = 0$.

\subsection{Eick's algebras}

Commutative cohomology may serve as another invariant helping to distinguish
algebras. In \cite{eick}, a computer-generated list of simple Lie algebras over
$\GF(2)$ of dimension $\le 20$ was presented, and sophisticated (nonlinear) 
methods were used to establish non-isomorphism of algebras in the list. 

For example, computer calculations with GAP show that the degree $2$ 
commutative cohomology with trivial coefficients of the two new 
$15$-dimensional simple Lie algebras in Eick's list, number 7 and 8, is of 
dimension $1$ and $2$ respectively. All the other ``conventional'' ``linear'' 
invariants of these two algebras we can think of (dimension of low-degree 
Chevalley--Eilenberg cohomology with trivial and adjoint coefficients, dimension of the $p$-envelope
and of the sandwich subalgebra, the absence of nondegenerate symmetric invariant forms) do coincide.

\section{Further questions}

Finally, we take a liberty to indicate some avenues for further research. Some 
of the questions listed here seem to be of a purely technical character, while 
others seem to be difficult and probably will require new nontrivial approaches.

\subsection{}
Is it possible to represent the commutative cohomology as a derived functor?
(This question seems to be tricky, as it is hard to imagine what the other
candidate for the role of the universal enveloping algebra in the commutative
case could be, see \S \ref{ss-u}).

\subsection{}
To compute commutative cohomology for various ``interesting'' algebras. In 
particular, for the three-dimensional simple Lie algebra, and for free Lie 
algebras.

\subsection{}
To get a formula relating $\Homol^2_{comm}(\Sl_n(A),K)$ and (a version of) 
cyclic cohomology of $A$ in the spirit of \cite{KL}. A glance at 
\cite[Proposition 2.1]{grishkov-zus} may suggest that the version of cyclic
cohomology, peculiar to characteristic $2$, which should appear here, is those
where the (skew)symmetric cochains are replaced by alternating ones.

\subsection{}
Establish an analog of the Hopf formula for the second degree commutative
\emph{homology} with trivial coefficients.

\subsection{}
Define the cup product in \S \ref{sec-cup} the same standard way as it is done 
for the Chevalley--Eilenberg cohomology and other classic cohomology theories, 
i.e. as a composition of the isomorphism provided by the K\"unneth formula, and
the map between cohomology of $L \oplus L$ and $L$ (see, for example, 
\cite[Exercise 7.3.8]{weibel}). For this, of course, we will need (a version of)
the K\"unneth formula for commutative cohomology. 

\subsection{}
The classical Stallings-Swan theorem says that groups of cohomological dimension
$1$ are free. In characteristics $0$ and $2$ it is an open question whether Lie
algebras of cohomological dimension $1$ are free. What about commutative Lie 
algebras? (Note that since we do not have a definition of commutative cohomology as a derived functor, the very notion of cohomological dimension
in this case is a bit problematic).

\subsection{}
It is well known (and easy to see) that the Euler-Poincar\'e characteristic of 
cohomology of a finite-dimensional Lie algebra, i.e. the alternating sum of 
dimensions of cohomology in all degrees, vanishes. The very notion of the 
Euler--Poincar\'e characteristic of the commutative cohomology does not make 
sense, as the sum
$$
\dim \Homol^0_{comm}(L,M) - \dim \Homol^1_{comm}(L,M) 
+ \dim \Homol^2_{comm}(L,M) - \dots
$$
is, generally, infinite and thus diverges. Can this sum be assigned a 
reasonable value using the theory of divergent series, similarly how it was 
(partially) done for cohomology of Lie superalgebras in \cite{euler}?

\subsection{}
As shown in \S \ref{ss-zass}, the dimension of the space of commutative 
$2$-cocycles with trivial coefficients on the Zassenhaus algebra 
$W_1^\prime(n)$, is equal to $2^n + n - 1$. Find a link with combinatorial 
interpretation of this number as the shortest length of a sequence of $0$ and 
$1$ containing all subsequences of length $n$ (see \cite[A052944]{eis}).

\subsection{}
Whether the variety of commutative Lie algebras is Schreier, i.e., whether a 
subalgebra of a free commutative Lie algebra is free?

\bigskip

Let us note at the end that recently Friedrich Wagemann has constructed a 
Hochschild--Serre-like spectral sequence for commutative cohomology. The 
construction more or less repeats the construction of the Hochschild--Serre 
spectral sequence for the Chevalley--Eilenberg cohomology.

\section*{Acknowledgements}

Thanks are due to Alexei Lebedev, Dimitry Leites, Ivan Shestakov, and 
Friedrich Wagemann for stimulating discussions and useful remarks. During the 
early stages of this work, Lopatkin enjoyed the hospitality of Czech Technical 
University in Prague, with special thanks to Pavel \v{S}\v{t}ov\'i\v{c}ek and 
\v{C}estm\'ir Burd\'ik; Zusmanovich enjoyed the hospitality of University of 
S\~ao Paulo. GAP \cite{gap} was utilized to check some computations performed in
this paper.

\appendix
\section*{Appendix. Algebraic discrete Morse theory}

Algebraic discrete Morse theory is an algebraic version of discrete Morse theory
developed independently by Sk\"oldberg, \cite{Sc}, and by J\"ollenbeck and 
Welker, \cite{JW}. It allows one to construct, starting from a chain complex, a
new homotopically equivalent smaller complex using directed graphs. Here, for 
the convenience of the reader, we present a short version of this machinery adapted
for cochain, rather than chain, complexes (this can be done formally by 
considering cochain complexes as chain complexes with negative indices and 
reverting arrows, but we prefer to write down everything explicitly). We follow
closely \cite[Chapter 2]{JW}, with minor simplifications and variations in 
notation.

Let 
$$
\C: \C_0 \overset{\dcobound_0}{\to} \C_1 \overset{\dcobound_1}{\to} \C_2 
\overset{\dcobound_2}{\to} \dots
$$
be a cochain complex of vector spaces over a field $K$ (which is assumed here to
be of arbitrary characteristic; in fact, the whole theory is generalized, with 
slight modifications, to the case of arbitrary complexes of free modules over an 
arbitrary associative ring). 

Let $X_n$ be a basis of the vector space $\C_n$. Write the differentials 
$\dcobound_n: \C_n \to \C_{n+1}$ with respect to these bases:
\[
\dcobound_n(c) = \sum\limits_{c^\prime \in X_{n+1}}[c:c^\prime]\cdot c^\prime ,
\]
where $c \in X_n$, and $[c:c^\prime]$ are coefficients from $K$.

From this data, we construct a directed weighted graph $\Gamma(\C) = (V,E)$. The
set of vertices $V$ of $\Gamma(\C)$ is the basis $V = \bigcup_{n \ge 0} X_n$, 
and the set $E$ of weighted edges consists of triples 
$$
\set{\>(c, c^\prime, [c:c^\prime])}
{c \in X_n, c^\prime \in X_{n+1}, [c:c^\prime] \ne 0\>} .
$$

A finite subset $M \subseteq E$ of the set of edges is called an 
\emph{acyclic matching}, if it satisfies the following two conditions:

\smallskip

(Matching) Each vertex $v \in V$ lies in at most one edge $e \in M$.

(Acyclicity) The subgraph $\Gamma^M(\C) = (V,E^M)$ of the graph $\Gamma(\C)$ has
no directed cycles, where
\[
E^M = (E \setminus M) \cup 
\set{\>(c',c, -\frac{1}{[c:c']})}{(c,c', [c:c']) \in M\>} .
\]

\smallskip

For an acyclic matching $M$ on the graph $\Gamma(\C)$, we introduce the 
following notation:
\begin{itemize}
\item[(1)]
Define
\[
X^M_n = \set{\> c \in X_n}{c \text{ does not lie in any edge } e \in M \>} .
\]

\item[(2)] 
Write $c^\prime \le c$ if $c \in X_n$, $c^\prime \in X_{n+1}$, and 
$[c : c^\prime] \ne 0$.

\item[(3)] 
$\mathrm{Path}(c,c')$ is the set of paths from $c$ to $c'$ in $\Gamma^M(\C)$.

\item[(4)] 
The weight $w(p)$ of a path 
$p = c_1 \to \ldots \to c_r \in \mathrm{Path}(c_1, c_r)$ is defined as
\[
w(c_1 \to \ldots \to c_r) = \prod\limits_{k=1}^{r-1} w(c_k \to c_{k+1})
\]
\[
w(c \to c^\prime) = \begin{cases} 
- \dfrac{1}{[c:c^\prime]} & \text{ if } c \le c^\prime \phantom{.} \\ 
\hskip 9pt [c:c^\prime]   & \text{ if } c^\prime \le c .
\end{cases}
\]
\end{itemize}

The following is a cohomological version of \cite[Theorem 2.2]{JW}:

\begin{theorem}
The cochain complex $(\C, \dcobound)$ is homotopy equivalent to the complex 
$(\C^M, \dcobound^M)$, where $\C_n^M$ is the vector space linearly spanned by 
$X_n^M$, and the differential $\dcobound_n^M: \C^M_n \to \C^M_{n+1}$ is defined
as
\[
\dcobound_n^M(c) = 
\sum\limits_{c^\prime \in X_{n+1}^M} 
\sum\limits_{p \in \mathrm{Path}(c,c^\prime)} w(p) \> c^\prime ,
\]
where $c \in \C^M_n$.
\end{theorem}

\renewcommand{\refname}{Software and online repositories}

\end{document}